\documentclass[11pt]{amsart}
\usepackage{mystyle}
\begin{document}
\title{Equivariant Chern Weil Forms and The Families Index Theorem}
\author[Richard Wedeen]{Richard Wedeen}
\address[Richard Wedeen]{Department of Mathematics,
   University of Texas at Austin,
   2515 Speedway,
   Austin TX, 78712, USA}
\email{rwedeen@math.utexas.edu}

\begin{abstract}
We apply the equivariance $\raw$ families principle to a geometric family of Clifford module bundles with an action of a compact Lie group $G$ to prove an equivariant version of Bismut's families index theorem on the differential Borel quotient of the geometric family.
\end{abstract}
\maketitle

In certain incarnations of the families index theorem, one considers a geometric family of Clifford module bundles over compact even dimensional Riemannian manifolds parametrized by a base $S$, i.e. an iterated bundle
\begin{equation}\label{iterated bundle}
   \mathscr E \raw X \raw S
\end{equation}
where $X\raw S$ is a fiber bundle with  compact even dimensional fibers, a horizontal distribution, and a metric on the vertical tangent bundle $T(X/S)$, and $\mathscr E \raw X$ is a $\Cl(T(X/S))$-module bundle with compatible connection.

This is the setting considered in \cite{Bismut:1986uh}, where Bismut adapted Quillen's ideas in \cite{QUILLEN198589} to construct a one-parameter family of superconnections
\begin{equation}
   \mathbb B^t: \Omega_S(\sH) \raw \Omega_S(\sH)
\end{equation}
on the infinite dimensional vector bundle $\sH$ of fiberwise sections of $\E \raw X$ over $S$ and proved that
\begin{equation}\label{index theorem}
   \lim_{t\raw 0} \ch(\mathbb B^t) = \int_{X/S} \hat A(\Omega^{X/S})\ch_{\mathbb \E/\mathbb S}(F^{\E/\mathbb S})
\end{equation}
where $\hat A(\Omega^{X/S})$ is the $\hat A$-genus of the vertical curvature and $\ch_{\mathbb \E/\mathbb S}(F^{\E/\mathbb S})$ is the relative Chern character of the twisting curvature of $\E$.

In this paper we consider the case when (\ref{iterated bundle}) has an action of a Lie group $G$ preserving the geometric data. A similar situation was considered in \cite{EquivDet} which studied the case of a $G$-action on a geometric family $X \raw S$ which induced a $G$-action on the determinant line bundle
\begin{equation}
   \mathcal L \raw S.
\end{equation}
There, the $G$-action was used to construct an iterated map of differential Borel quotients
\begin{equation}
   (\mathcal L_G)_\nabla \raw (S_G)_\nabla \raw B_\nabla G
\end{equation}
over the simplicial sheaf $B_\nabla G$ of principal $G$-bundles with connection introduced and studied in \cite{BNablaG}. The $G$-equivariant connection on $\mathcal L \raw S$ from \cite{Bismut:1986wl} was used to construct a connection on $(\mathcal L_G)_\nabla \raw (S_G)_\nabla$ from which the moment map and curvature were computed. The curvature was identified with the $G$-equivariant curvature on $\mathcal L$ using the identification of $\Omega_{(X_G)_\nabla}$ with the Weil-model for $X$ \cite{BNablaG}.

We adapt these ideas and consider a geometric family of differential Borel quotients
\begin{equation}
 (\E_G)_\nabla \raw (X_G)_\nabla \raw (S_G)_\nabla
\end{equation}
fibered over $B_\nabla G$. From this we construct a family of superconnections $\mathbb B^t_G$ on a bundle $(\sH_G)_\nabla \raw (S_G)_\nabla$ whose curvature is an equivariant form in the Weil-model for $S$. We also construct equivariant curvatures $\Omega_G^{X/S}$ and $F_G^{\mathscr E/\mathbb S}$ on $(X_G)_\nabla$. We then prove our main theorem
\begin{equation}
   \lim_{t\raw 0}\ch(\mathbb B^t_G) = \int_{X/S} \hat A(\Omega^{X/S}_G) \ch_{\E/\mathbb S}(F^{\E/\mathbb S}_G).
\end{equation}

An outline of the paper is as follows. In \S \ref{Section 2} we set up the necessary differential geometry on manifolds and generalized manifolds, i.e. simplicial sheaves on the site of manifolds. In particular we construct the connection and curvature of the differential Borel quotient of a vector bundle with $G$-action and $G$-equivariant connection. In \S \ref{Section 3} we review the families index theorem as proved by Bismut following \cite{berline2003heat}. Finally in \S \ref{Section 4} we construct the superconnection $\mathbb B^t_G$ and prove the main result.

I thank Dan Freed, Charlie Reid, and Dan Weser for helpful discussions and the mathematics department at UT Austin that made this paper possible.
\section{Preliminaries}\label{Section 2}

\subsubsection{Principal Bundles}

This section is a summary of the discussion appearing in \cite{EquivDet}. Let $G,H$ be Lie groups and let $\pi: P \raw X$ be a principal $H$-bundle with a $G$-action. We endow $\pi$ with a $G$-invariant connection, i.e. a horizontal distribution $W \subset TP$ that is invariant under the $G$-action on the left and the $H$-action on the right.

Let $r: Q\raw M$ be a principal $G$-bundle.

\begin{notation} Given a $G$-space $Y$ we can form the associated bundle
\begin{equation}
   Y_Q : = Q \times_G Y.
\end{equation}
We will denote by $q: Q \times Y \raw Y_Q$ the quotient and by $p: Q \times Y \raw Y$ the projection. Finally, if $E \raw Y$ is a vector bundle, we will write 
\begin{equation}
\underline E:= p^\ast E
\end{equation}
to denote the pullback vector bundle of $E$ over $Q\times Y$.
\end{notation}

The map $Q \times P \xrightarrow{id \times \pi} Q \times X$ is $G$-equivariant and descends to a map $\pi_Q: P_Q \raw X_Q$ which is a principal $H$-bundle.
Suppose given a connection on $r$, i.e. a $G$-invariant horizontal distribution 
\begin{equation}\label{V distribution}
V \subset TQ.
\end{equation} 
We use the given data to construct a connection on $\pi_Q$.
\begin{construction}\label{distribution}
The distribution $V \oplus W \subset TQ \oplus TP$ is $G$-invariant and descends to an $H$-invariant horizontal distribution $q_\ast(V \oplus W) \subset TP_Q$.
\end{construction}

Denote by $\g, \h$ the Lie algebras of $G,H$ and choose a basis $(e_a)$ of $\g$. 
Let $\Phi \in \Omega^1_Q(\g)$ be the connection 1-form on $Q$ with $\ker(\Phi) = V$. We will write $\Phi := \phi^a e_a$ where the coefficients $\phi^a \in \Omega^1_Q$ are $G$-invariant. Let $\Theta \in \Omega^1_P(\h)$ be the $G$-invariant connection 1-form on $P$ satisfying $\ker(\Theta) = W$. 
We set $\omega:= \omega^a e_a \in \Omega^2_Q(\g)$ to be the curvature of $\Phi$ and $\Omega \in \Omega^2_P(\h)$ the curvature of $\Theta$.

Finally, let $\Theta_Q \in \Omega^1_{P_Q}(\h) \simeq \Omega^1_{Q \times P}(\h)_{\basic}$ be the connection 1-form on $\pi_Q$ with $\ker(\Theta_Q) = q_\ast(V\oplus W)$ and let $\Omega_Q \in \Omega^2_{P_Q}(\h) \simeq \Omega^2_{Q\times P}(\h)_{\basic}$ denote its curvature.

\begin{lemma}\label{associated curvature}
\begin{align}
\Theta_Q &= \Theta - \phi^a\iota_a\Theta \label{Q connection}\\
\Omega_Q = \Omega - \phi^a\iota_a&\Omega + \frac{1}{2}\phi^a\phi^b\iota_b\iota_a\Omega - \omega^a\iota_a\Theta \label{Q curvature}
\end{align}
\end{lemma}
See \cite{EquivDet} for a proof of this calculation.

\subsubsection{Associated Bundles}

We now consider the case where $\pi: P \raw X$ is the principal $H:= GL_N(\C)$ bundle of frames for a complex vector bundle $E \raw X$ of rank $N$ with a $G$-action. 
This action induces a left $G$-action on $\pi$ that commutes with the $H$-action on the right.

Pulling back $E$ along $\pi$ gives the trivial bundle of rank $N$ over the bundle of frames
\begin{equation}\label{untwisting}
\begin{tikzcd}
   \underline \C^N \arrow[r] \arrow[d] & E \arrow[d] \\
   P \arrow[r, "\pi"] & X
\end{tikzcd}
\end{equation}
This implies that sections of $E$ can be expressed as $H$-invariant functions on $P$ valued in $V$, and more generally,
\begin{equation}\label{form id}
   \Omega_X(E) \simeq \Omega_P(\C^N)_{\basic}
\end{equation}
where the basic forms are those that are horizontal and $H$-invariant under the action of $H$ on both $P$ and $\C^N$.

Suppose given a $G$-equivariant covariant derivative 
\begin{equation}
\nabla: \Omega^0_X(E) \raw \Omega^1_X(E)
\end{equation}
on $E$.
Using the identification (\ref{form id}), $\nabla$ becomes a derivation on $\Omega_P(\C^N)_{\basic}$ and can be expressed as
\begin{equation}\label{untwisted derivative}
\nabla = d + \Theta
\end{equation}
where $\Theta \in \Omega^1_P(\gl_N)$ is a $G$-invariant connection 1-form on $P$ thought of as a matrix of 1-forms. Applying \ref{associated curvature}, we can construct a connection 1-form $\Theta_Q \in \Omega^1_{P_Q}(\gl_N)$ on $\pi_Q: P_Q \raw X_Q$.

We now proceed to construct the covariant derivative on $E_Q$ associated to this connection. First, we note that there are identifications
\begin{equation}\label{associated id}
\begin{split}
\Omega_{X_Q}(E_Q) &\simeq \Omega_{P_Q}(\C^N)_{\bas_H}\\
&\simeq \Omega_{Q\times P}(\C^N)_{\bas_{G\times H}}\\
&\simeq \Omega_{Q \times X}(\underline E)_{\bas_{G}}.
\end{split}
\end{equation}

\begin{lemma}\label{Q covariant derivative}
The connection $\Theta_Q$ induces a covariant derivative on $E_Q \raw X_Q$, expressed under the identification (\ref{associated id}) as a derivation $\nabla_Q: \Omega_{Q\times X}(\underline E)_{\bas} \raw \Omega_{Q\times X}(\underline E)_{\bas}$, given by
\begin{equation}
   \nabla_Q := \nabla - \phi^a\tau_a.
\end{equation}
where $\nabla$ is understood to be the covariant derivative on $\underline E$ obtained by pulling back the covariant derivative on $E$ by $p$.
Here, $\tau_a \in \End(\underline E)$ is the endomorphism
\begin{equation}\label{endomorphism}
\tau_a := \mathcal L_{\xi_a} - \nabla_{\xi_a}
\end{equation} 
where $\xi_a$ is the vector field generated by the action of $e_a$ on $Q \times X$and $\mathcal L_{\xi_a}$ is the Lie derivative on $\Omega_{Q\times X}(\underline E)$.
\end{lemma}

\begin{proof}
As a derivation on $\Omega_{Q\times P}(\C^N)_{\bas_{G\times H}}$, the covariant derivative is expressed as
\begin{equation}
\begin{split}
   \nabla_Q &= d + \Theta_Q\\
      &= d_Q + d_P + \Theta - \phi^a\iota_a\Theta.
\end{split}
\end{equation}
where $d_P, d_Q$ are the de Rham differentials on $P$ and $Q$.
The term $d_Q + d_P + \Theta$ on $\Omega_{Q\times P}(\C^N)$ becomes identified with $d_Q + \nabla = p^\ast \nabla$ on $\Omega_{Q\times X}(\underline E)$ and the term $\iota_a\Theta \in \Omega^0_{Q\times P}(\gl_N)$ is identified with $\tau_a \in \End(\underline E)$.
\end{proof}

The following lemma is standard and we state it without proof.

\begin{lemma}\label{Q curvature identification}
Under the identification $\Omega_{P_Q}(\gl_N)_{\bas_{H}} \simeq \Omega_{X_Q}(\End(E_Q))$ the curvature of $\nabla_Q$ is given by (\ref{Q curvature}), i.e.
\begin{equation}
   \nabla_Q^2 = \Omega_Q
\end{equation}
\end{lemma}

\subsubsection{Generalized Manifolds}

We now enter the category of generalized manifolds with the aim of constructing a universal connection and covariant derivative. We first review some background on simplicial sheaves and refer the reader to \cite{BNablaG} for proofs and details.

Let $\Omega^k : \text{Man}^{op} \raw \Delta\text{-Set}$ be the simplicial sheaf that assigns a test manifold $M$ the set $\Omega^k_M$ of $k$-forms on $M$. The \emph{simplicial de Rham complex} \cite{BNablaG} is defined to be the simplicial sheaf
\begin{equation}\label{de Rham}
\Omega := \big(\bigoplus_{k\geq 0} \Omega^k , d\big)
\end{equation}
which assigns an $n$-dimensional test manifold $M$ the differential graded algebra of forms $\Omega_M$ on $M$.

We will have reason to consider formal series of forms in the completion of (\ref{de Rham}) and we review the construction here. For $k \geq 0$ let $\Omega^{\geq k} \subset \Omega$ be the ideal generated by forms of degree $k$. There is an inverse system
\begin{equation}
\cdots \raw \Omega/\Omega^{\geq k+1} \raw \Omega/\Omega^{\geq k} \raw \cdots \raw \Omega/\Omega = 0
\end{equation}
The completion is defined as the limit
\begin{equation}\label{completion}
\begin{split}
   \overline \Omega := \lim_{\law} \left(\cdots \raw \Omega/\Omega^{\geq k+1} \raw \Omega/\Omega^{\geq k} \raw \cdots\right)
\end{split}
\end{equation}
and can be identified with the direct product
\begin{equation}
\overline \Omega \simeq \prod_{k\geq 0} \Omega^k.
\end{equation}
We will later use the completion to consider forms of unbounded degree on generalized manifolds. For now, we note that for a finite dimensional manifold $M$, the completion is isomorphic to the ordinary de Rham complex
\begin{equation}
\overline \Omega_M \simeq \Omega_M.
\end{equation} 
The Koszul complex associated to a vector space $V$ is the differential graded algebra generated by $\Lambda^1 V$. It is given by
\begin{equation}
\Kos(V) := \Lambda V \otimes \Sym(V)
\end{equation}
with differential
\begin{equation}
d_K(v) = \tilde v, \hspace{.5cm} d_K(\tilde v) = 0
\end{equation}
where $v \in \Lambda^1 V$ and $\tilde v \in \Sym^1(V)$ is $v$ regarded as an element of the symmetric algebra of total degree 2. As in (\ref{completion}) we let $\Kos^{\geq k}$ be the ideal generated by elements of degree $k$ and we define the \emph{completed Koszul algebra} to be
\begin{equation}
\overline \Kos(V) := \lim_{\law} \big(\cdots \raw \Kos(V)/\Kos^{\geq k+1}(V) \raw \Kos(V)/\Kos^{\geq k}(V) \raw \cdots \big)
\end{equation}
which can be identified with
\begin{equation}
\overline \Kos(V) \simeq \Lambda V \otimes \overline \Sym(V)
\end{equation}
where $\overline \Sym(V)$ is the algebra of formal power series on $V^*$.

We now recall some theorems from \cite{BNablaG} and their corresponding statements on completions for future use.
\begin{theorem}[\cite{BNablaG} Theorem 8.1]
Let $V$ be a finite dimensional real vector space. Then there is an isomorphism 
\begin{equation}\label{kos iso}
\eta: \Kos(V^*) \xrightarrow{\sim} \Omega_{\Omega^1 \otimes V}
\end{equation}
 defined on $\ell \in \Lambda^1(V^\ast)$ by
\begin{equation}\label{kos iso 2}
\begin{split}
\eta(\ell): \Omega^1 \otimes &V \raw \Omega^1 \\
\alpha^i v_i &\mapsto \langle \ell, v_i \rangle \alpha^i 
\end{split}
\end{equation}
and extended to $\Kos(V^\ast)$ by requiring that it be a morphism of differential graded algebras.
\end{theorem}

\begin{corollary}
The isomorphism (\ref{kos iso}) extends to an isomorphism of completions
\begin{equation}\label{compl.iso}
\overline \Omega_{\Omega^1 \otimes V} \simeq \overline \Kos(V^*).
\end{equation}
\end{corollary}

Let $E_\nabla G$ be the simplicial sheaf of principal $G$ bundles with connection and trivialization. We remark that pulling down the connection form along the trivializing section gives a weak equivalence 
\begin{equation}\label{w.e.E_nablaG}
E_\nabla G \simeq \Omega^1 \otimes \g.
\end{equation}
Combining (\ref{compl.iso}) and (\ref{w.e.E_nablaG}) gives
\begin{corollary}
There is an identification
\begin{equation}
\overline \Omega_{E_\nabla G} \simeq \overline \Kos(\g^\ast)
\end{equation}
where $\overline \Kos(\g^\ast) := \Lambda \g^\ast \otimes \overline \Sym(\g^\ast)$ is the completed Weil algebra.
\end{corollary}

Similarly, for a smooth manifold $Y$ we have

\begin{proposition}[\cite{BNablaG} Corollary 8.24]\label{tensor}
There is an isomorphism
\begin{equation}
\Omega_{E_\nabla G \times Y} \simeq \Kos(\g^*) \otimes \Omega_Y
\end{equation}
which extends to an isomorphism of completions
\begin{equation}
\overline \Omega_{E_\nabla G \times Y} \simeq \overline \Kos(\g^*) \otimes \Omega_Y
\end{equation}
\end{proposition}

We recall that the groupoid sheaf of principal $G$ bundles with connection is weakly equivalent to the simplicial sheaf
\begin{equation}\label{BnablaGaction}
\xymatrix{ B_\nabla G \hspace{.2cm} \simeq \hspace{.2cm}\ \big(E_\nabla G &  E_\nabla G \times G \ar@<.6ex>[l] \ar@<-.6ex>[l] &  \cdots \big) \ar@<1ex>[l] \ar@<0ex>[l] \ar@<-1ex>[l]}
\end{equation}
where the 1-simplices are determined by the $G$-action on $E_\nabla G$ and the higher simplices by group composition.

For a $G$-space $Y$, one can use (\ref{BnablaGaction}) to define the differential Borel quotient as the simplicial sheaf
\begin{equation}
\xymatrix{ (Y_G)_\nabla \hspace{.2cm} := \hspace{.2cm}\ \big(E_\nabla G \times Y &  E_\nabla G \times Y \times G \ar@<.6ex>[l] \ar@<-.6ex>[l] &  \cdots\big) \ar@<1ex>[l] \ar@<0ex>[l] \ar@<-1ex>[l]}
\end{equation}
constructed using the $G$ action on $E_\nabla G \times Y$.

The action of $G$ on $E_\nabla G$ in (\ref{BnablaGaction}) induces two derivations on the de Rham complex $\Omega_{E_\nabla G}$. The \emph{contraction} (\cite{BNablaG} Lemma 7.25) in the direction $\xi \in \g$ is a degree (-1) derivation 
on $\Omega_{E_\nabla G} \simeq \Kos(\g^*)$ which extends to a derivation on $\overline \Omega_{E_\nabla G} \simeq \overline \Kos(\g^*)$.
If we choose bases $\theta^a \in \Lambda^1\g^*$ and $z^a := d\theta^a \in \Sym^1(\g^*)$ dual to $e_a \in \g$, we can express this as
\begin{equation}
\iota_{e_b}\theta^a := \delta^a_b,  \hspace{.5cm} \iota_{e_b}z^a := -f^a_{bc}\theta^c
\end{equation}
where $f^a_{bc}:= \langle e^a, [e_b,e_c]\rangle$.

The \emph{Lie derivative} in the direction $\xi \in \g$ is a degree 0 derivation $\mathcal L_{\xi}$ on $\Kos(\g^*)$ which extends to $\overline \Kos(\g^*)$ and is induced from the coadjoint action of $G$ on $\g^\ast$. It is given by
\begin{equation}
\mathcal L_{e_b}\theta^a := -f^a_{bc}\theta^c,  \hspace{.5cm} \mathcal L_{e_b}z^a := -f^a_{bc}z^c.
\end{equation}
These satisfy Cartan's relation
\begin{equation}
\mathcal L_\xi = d\iota_\xi + \iota_\xi d.
\end{equation}

Similarly, there are two derivations $\iota, \mathcal L$ on the completed de Rham complex of $\overline \Omega_{E_\nabla G  \times Y} \simeq \Kos(\g^*)\otimes \Omega_Y$. They act on $\overline \Kos(\g^*)$ by the derivations given above and by contraction and Lie derivative on $\Omega_Y$.
Finally define the subcomplexes \emph{basic forms}
\begin{equation}
(\overline \Omega_{E_\nabla G})_{\bas} \subset \overline \Omega_{E_\nabla G}
\end{equation} 
\begin{equation}
(\overline \Omega_{E_\nabla G \times Y})_{\bas} \subset \overline \Omega_{E_\nabla G \times Y}
\end{equation}
to be those elements annilhilated by $\mathcal L$ and $\iota$. 

\begin{theorem}[\cite{BNablaG} Theorem 7.20, 7.28]
There are identifications
\begin{equation}
\Omega_{B_\nabla G} \simeq (\Omega_{E_\nabla G})_{\bas} \simeq ( \Kos(\g^*))_{\bas} \simeq \Sym(\g^*)^G
\end{equation}
\begin{equation}
\Omega_{(Y_G)_\nabla} \simeq (\Kos(\g^*) \otimes \Omega_Y)_{\bas}
\end{equation}
and they extend to identifications of their respective completions.
\end{theorem}

Let $r: Q \raw M$ be a principal $G$-bundle with connection $\Phi \in \Omega^1_Q(\g)$ and curvature $\omega \in \Omega^2_Q(\g)$. We express $\Phi = \phi^a e_a$ and $\omega = \omega^a e_a$ where $\phi^a, \omega^a$ are $G$-invariant forms on $Q$. Let 
\begin{equation}\label{classifying map}
f: M \raw B_\nabla G
\end{equation} be the classifying map which induces the pullback

\begin{equation}
\begin{tikzcd}
Q \arrow[r, "\tilde f"] \arrow[d] & E_\nabla G \arrow[d] \\
M \arrow[r, "f"] & B_\nabla G
\end{tikzcd}
\end{equation}
where $\tilde f$ corresponds to the form $\Phi$ under the identification $E_\nabla G(Q) \simeq \Omega^1_Q(\g)$.

Using the map (\ref{kos iso 2}) we see that 
\begin{equation}
Q \xrightarrow{\tilde f} E_\nabla G \xrightarrow{\eta(\theta^b)} \Omega
\end{equation}
is the element of $\Omega_Q$ given by
$\langle \theta^b, e_a\rangle \phi^a = \phi^b$.  Pulling back forms along $\tilde f$ gives the Chern Weil map 
\begin{equation}
\begin{split}
   \tilde f^\ast: \overline \Kos(\g^*) &\raw \Omega_Q\\
   \theta^a &\mapsto \phi^a.
\end{split}
\end{equation}

The map $\tilde f^\ast$ restricts to a morphism of differential graded algebras on $\Kos(\g^\ast) \subset \overline \Kos(\g^\ast)$ and this defines the rest of the map. In particular, if we set 
\begin{equation}\label{chi}
\chi^a := z^a + \frac{1}{2}f^a_{bc}\theta^b\theta^c 
\end{equation}
in degree 2 of $\Kos(\g^\ast) \subset \overline \Kos(\g^\ast)$, we see that
\begin{equation}
\tilde f^\ast : \chi^a \mapsto \omega^a.
\end{equation}

We also note that
\begin{equation}
\tilde f \times id: Q \times Y \raw E_\nabla G \times Y
\end{equation}
is $G$-equivariant and descends to a map
\begin{equation}
f_Y : Y_Q \raw (Y_G)_\nabla
\end{equation}
which induces a map on the level of forms
\begin{equation}\label{pullback}
\begin{split}
   f_Y^\ast : (\overline \Kos(\g^*) \otimes \Omega_Y)_{\bas} &\raw (\Omega_{Q\times Y})_{\bas}\\
   \theta^a \otimes \alpha &\mapsto \phi^a\wedge\alpha\\
   \chi^a \otimes \alpha &\mapsto \omega^a\wedge\alpha.
\end{split}
\end{equation}

Finally, let $E \raw Y$ be a vector bundle on $Y$ with a $G$ action. Then (\ref{pullback}) extends to a map on forms valued in sections of $E$
\begin{equation}\label{pullback sections}
   f_Y^\ast: (\overline \Kos(\g^*)\otimes \Omega_Y(E))_{\bas} \raw \Omega_{Y_Q}(E_Q)
\end{equation}
by pulling back sections along $p: Q\times Y \raw Y$.

\subsubsection{Covariant Derivative on Generalized Manifolds} In this section we construct a covariant derivative on the differential Borel quotient of a vector bundle. As before, let $E \raw X$ be a complex vector bundle with a $G$-action and a $G$-equivariant covariant derivative $\nabla$. The map $E_\nabla G \times E \xrightarrow{id\times \pi} E_\nabla G \times X$ is $G$-equivariant and descends to a map of Borel quotients
\begin{equation}\label{pi_G}
   \pi_G: (E_G)_\nabla \raw (X_G)_\nabla.
\end{equation}
We will denote the bundle $E_\nabla G \times E$ by $\underline E$ and we observe that it can be written as a pullback of simplicial sheaves
\begin{equation}
\begin{tikzcd}
\underline E \arrow[r]\arrow[d] & (E_G)_\nabla \arrow[d, "\pi_G"]\\
E_\nabla G \times X \arrow[r]  & (X_G)_\nabla
\end{tikzcd}
\end{equation}
We define forms on $(X_G)_\nabla$ valued in sections of $(E_G)_\nabla$ by
\begin{equation}\label{id}
\begin{split}
\overline \Omega_{(X_G)_\nabla}((E_G)_\nabla) 
   &:= \overline \Omega_{E_\nabla G \times X}(\underline E)_{\bas}\\ 
   &\simeq (\overline \Omega_{E_\nabla G} \otimes \Omega_X(E))_{\bas}\\
   &\simeq (\overline \Kos(\g^\ast) \otimes \Omega_X(E))_{\bas}.
\end{split}
\end{equation}
where we have used \ref{tensor} in the second line and the grading is given by the total degree.

There is a covariant derivative on $\Omega_{(X_G)_\nabla}((E_G)_\nabla)$ over the algebra $\Omega_{(X_G)_\nabla}$ given by
\begin{equation}\label{G covariant derivative}
\nabla_G := d_K + \nabla - \theta^a\tau_a.
\end{equation}
Its curvature
\begin{equation}
\Omega_G := \nabla_G^2
\end{equation}
is a basic endomorphism valued 2-form, i.e. $F_G \in (\overline\Kos(\g^\ast)\otimes \Omega_X(\End(E)))_{\bas}$. 

\begin{lemma}\label{pullback equivariant curvature}
The equivariant curvature satisfies
\begin{equation}
f_X^\ast \Omega_G = \Omega_Q
\end{equation}
where $\Omega_Q$ is (\ref{Q curvature}) and $f_X^*$ is (\ref{pullback sections}) applied to the bundle $\End(E) \raw X$.
\end{lemma}

\begin{proof}
Let $\bar\nabla_Q: \Omega_{Q\times X}(\underline E) \raw \Omega_{Q\times X}(\underline E)$ be the extension of $\nabla_Q: \Omega_{Q\times X}(\underline E)_{\basic} \raw \Omega_{Q\times X}(\underline E)_{\basic}$  using  formula (\ref{G covariant derivative}). Similarly, let $\bar \nabla_G$ be the extension of $\nabla_G$ to $\overline \Kos(\g^*) \otimes \Omega_X(\End(E))$ to all basic and nonbasic forms. Applying definitions shows that

\begin{equation}\label{universal pullback1}
\begin{tikzcd}
   \overline \Kos(\g^\ast) \otimes \Omega_X(E) \arrow[r, "\bar \nabla_G"] \arrow[d, "f_X^\ast"] & \overline \Kos(\g^\ast) \otimes \Omega_X(E) \arrow[d, "f_X^\ast"]\\
   \Omega_{Q \times X}(\underline E) \arrow[r, "\bar \nabla_Q"] & \Omega_{Q\times X}(\underline E)
\end{tikzcd}
\end{equation}
commutes.
which implies that
\begin{equation}
\begin{tikzcd}
   \Omega^0_X(E) \arrow[r, "\bar\nabla_G^2"] \arrow[d, "\tilde f_X^\ast = p^\ast"] & \overline \Kos(\g^\ast) \otimes \Omega_X(E) \arrow[d, "\tilde f_X^*"]\\
   \Omega^0_{Q\times X}(\underline E) \arrow[r, "\bar\nabla_Q^2"] & \Omega_{Q\times X}(\underline E)
\end{tikzcd}
\end{equation}
commutes.

Since sections of $\underline E \simeq p^\ast E$ are generated over the ring of functions $\Omega^0_{Q\times X}$ by $p^\ast\sigma$ for $\sigma \in \Omega^0_X(E)$, we see that $\bar\nabla_Q^2 = \tilde f_X^\ast \bar\nabla_G^2$ which implies $\nabla_Q^2 = f_X^\ast \nabla_G^2$. Applying Lemma \ref{Q curvature identification} finishes the proof.
\end{proof}

\section{Review of The Families Index Theorem}\label{Section 3}

We now recall the statement of the families index theorem \cite{Bismut:1986uh}. This section is based on \cite{berline2003heat} to which we refer the reader for a more thorough exposition.

\subsubsection{The Bismut Superconnection}

\begin{definition}
A \emph{geometric family} is
   \begin{itemize}
      \item A smooth fiber bundle $X \xrightarrow{\pi} S$ with compact even dimensional fibers
      \item A horizontal distribution $H \subset TX$; i.e. $\pi$ induces an isomorphism on the fibers $\pi_\ast: H_x \xrightarrow{\sim} TS_{\pi(x)}$
      \item A metric $g^{X/S}$ on $T(X/S) := \ker(\pi_\ast)$.
   \end{itemize}
\end{definition}

Let $(X \xrightarrow{\pi} S, H, g^{X/S})$ be a geometric family. We construct a connection $\nabla^{X/S}$ on $T(X/S)$ as follows. 

\begin{construction}\label{vertical connection} 
Choose a metric $g^S$ on $TS$ and pull it back to a metric $\pi^\ast g^S$ on $H$. 
The metric $g^X := g^{X/S} \oplus \pi^\ast g^S$ induces a Levi-Civita connection $\nabla^{g^X}$ on $TX$. 
The horizontal distribution induces a split short exact sequence of bundles
\begin{equation}\label{split exact}
0 \raw H \mono TX \stackrel[\iota]{P}{\rightleftarrows} T(X/S) \raw 0.
\end{equation} 
Define $\nabla^{X/S} := P \circ \nabla^{g^X} \circ \iota$, i.e. the compression of the Levi-Civita connection to the vertical tangent space.
\end{construction}

\begin{remark}
The proof that this construction is independent of the choice of metric $g^S$ is straightforward and can be found for example in \cite{berline2003heat}.
\end{remark}

\begin{remark}
The curvature of $\nabla^{X/S}$
\begin{equation}\label{vertical curvature}
\Omega^{X/S} := (\nabla^{X/S})^2
\end{equation}
is a 2-form in  $\Omega_X(\so(T(X/S)))$.
\end{remark}

\begin{definition}
A \emph{geometric Clifford family} is
   \begin{itemize}
      \item A geometric family $(X \xrightarrow{\pi} S, H, g^{X/S})$
      \item A $\Z/2$-graded $\Cl(T(X/S), g^{X/S})$-module bundle $\mathscr E \raw X$
      \item A \emph{Clifford connection} $\nabla^{\mathscr E}$ on $\mathscr E$ satisfying the Leibniz rule for Clifford multiplication: for $v$ a section of $T(X/S) \subset \Cl(T(X/S))$ and $\sigma \in C^\infty(\mathscr E)$ 
      \begin{equation}\nabla^{\mathscr E} (v\cdot \sigma) = \nabla^{X/S}v \cdot \sigma + v\cdot \nabla^{\mathscr E} \sigma.\end{equation}
   \end{itemize}
\end{definition}

\begin{remark}
The $\Z/2$-grading $\E \simeq \E^+ \oplus \E^-$ is an eigen-decomposition for the action of the \emph{chirality operator}, $\Gamma \in C^\infty(\Cl(TX/S))$, defined locally as
\begin{equation}\label{chirality operator}
\Gamma := i^{n/2}e_1\cdots e_n
\end{equation}
where $(e_i)$ is a local orthonormal basis of $T(X/S)$. It satisfies $\Gamma^2 = 1$ and its $\pm 1$ eigenbundle is $\E^{\pm}$.
\end{remark}

Let $(\mathscr E \raw X \xrightarrow{\pi} S, H, g^{X/S}, \nabla^{\mathscr E})$ be a geometric Clifford family and let $\mathscr H \raw S$ be the $\Z/2$-graded infinite dimensional vector bundle whose fiber over $z \in S$ is the space of $C^\infty$ sections of $\mathscr E$ over $X_z := \pi^{-1}(z)$. We can identify its space of sections as $\Gamma(S;\mathscr H):= \C^\infty(X; \mathscr E)$ and we define forms valued in $\mathscr H$ to be
\begin{equation}
\Omega_S(\mathscr H) := C^\infty(X; \Lambda H^* \otimes \mathscr E)
\end{equation}
where we are using that $H$ is horizontal to identify $\pi^\ast \Lambda T^\ast S \simeq \Lambda H^*$. There is a $\Z/2$-grading on $\E$ and a $\Z$-grading on $\Lambda H^*$ and we use the total (mod 2) to endow $\Omega_S(\sH)$ with a $\Z/2$-grading.

A \emph{superconnection} on $\sH$ is a differential operator $\mathbb A$ on $\Omega_S(\sH)$ of odd degree such that
\begin{equation}
\mathbb A(\nu.\sigma) = d_S\nu \cdot \sigma + (-1)^{|\nu|}\nu \cdot \mathbb A \sigma
\end{equation}
where $\nu \in \Omega_S$ and $\sigma \in \Gamma(S; \sH)$.  There is a decomposition $\mathbb A = \sum_i \mathbb A_{[i]}$ where
\begin{equation}
   \mathbb A_{[i]}: \Omega^\bullet_S(\sH) \raw \Omega^{\bullet +i}_S(\sH)
\end{equation}
raises the total degree by $i$.

It follows from the definition that $\mathbb A_{[1]}$ is a connection on $\sH$. The other components $\mathbb A_{[i\neq 1]}$ supercommute with the action of the algebra of forms $\Omega_S$ which shows that $\mathbb A_{[i\neq 1]}$ acts by exterior multiplication by forms on $S$ valued in vertical differential operators. More precisely, let 
\begin{equation}\label{differential operators}
\mathcal D(\E)\raw S
\end{equation} be the bundle whose fiber over $z \in S$ is the space of differential operators on $\Gamma(X_z ; \E_z)$ where $\E_z$ is the restriction of $\E$ to $X_z$; a section of $\mathcal D(\E)$ is a smooth family of vertical differential operators on $\E$. Then
\begin{equation}
\mathbb A_{[i\neq 1]} \in \Omega^i_S(\mathcal D(\E)).
\end{equation}

There is a specific superconnection constructed from the data of a geometric Clifford family, the \emph{Bismut superconnection}, which has components up to degree 2 and is defined as
\begin{equation}
\begin{split}
\mathbb B &= \mathbb B_{[0]} + \mathbb B_{[1]} + \mathbb B_{[2]}\\
&:= \slashed D + \left(\tilde \nabla^\E + \frac{1}{2}k\right) - \frac{1}{4}c(T).
\end{split}
\end{equation}

The terms in the formula are as follows:
\begin{itemize}
   \item 
      $\slashed D \in \Omega^0_S(\mathcal D(\E))$ is the vertical Dirac operator given locally by 
         \begin{equation}
         \slashed D = \sum_i c(e_i)\nabla_{e_i}: C^\infty(\mathscr E) \raw C^\infty(\mathscr E)
         \end{equation}
      where $(e_i)$ is a local orthonormal basis of $T(X/S)$. 
   \item 
      For $\xi \in \Gamma(TS)$ and $\sigma \in \Omega^0_S(\mathscr H)$, we define 
         \begin{equation}\label{tilde nabla}
         \tilde \nabla^{\E}_\xi \sigma := \nabla^{\mathscr E}_{\tilde \xi} \sigma
         \end{equation} 
      where $\tilde \xi \in \Gamma(H)$ is the horizontal lift of $\xi$. 
   \item 
      $T \in C^\infty(\Lambda^2 H^* \otimes T(X/S))$ is the Frobenius tensor defined by 
         \begin{equation}\label{frobenius}
         T(\eta, \zeta) = -P[\eta, \zeta]
         \end{equation}
      for $\eta, \zeta \in \Gamma(H)$, where $P$ is the projection onto the vertical defined in (\ref{split exact}).
   \item
      $c(T) \in C^\infty(\Lambda^2 H^* \otimes \End(\E)) \subset \Omega^2_S(\mathcal D(\E))$
      is obtained by applying the Clifford map 
         \begin{equation}\label{clifford map}
         c: \Cl(T(X/S)) \raw \End(\mathscr E)
         \end{equation} 
      to the Frobenius tensor $T$.
   \item 
      $k \in C^\infty(H^*) \subset C^\infty(H^* \otimes \End(\E)) \subset \Omega^1_S(\mathcal D(\E))$ is multiplication by the mean curvature tensor of the fibers. Using the horizontal distribution, both $k$ and the vertical volume form $\nu^{X/S} \in C^\infty(\Lambda^n T^\ast(X/S))$ can be identified with forms of degree 1 and $n$ in $\Omega_X$. Then the mean curvature form is characterized by the equation
         \begin{equation}\label{mean curvature}
            d\nu^{X/S} = k \wedge \nu^{X/S} + \sum_{i=1}^n \langle T, e_i \rangle \iota_{e_i}\nu^{X/S}.
         \end{equation}
\end{itemize}

\begin{remark}
$\slashed D$ is valued in vertical differential operators of order 1 on $\E$ while $k, c(T)$ are valued in vertical differential operators of order 0.
\end{remark}

The curvature of the Bismut superconnection 
\begin{equation}
\mathcal F := \mathbb B^2
\end{equation}
is an element of $\Omega_S(\mathcal D(\E))$. We express it as
   \begin{equation}
   \mathcal F = \slashed D^2 + \mathcal F^+
   \end{equation}
where $\slashed D^2 \in \Omega^0_S(\mathscr D(\E))$ is the Dirac Laplacian and $\mathcal F^+ \in \Omega^{i\geq 1}_S(\mathcal D(\E))$ is a finite sum of forms valued in vertical differential operators of order 1. 

\begin{remark}\label{dirac laplacian}
The heat operator $e^{-t\slashed D^2}$ of the Dirac Laplacian has a smoothing kernel and is therefore trace-class due to compactness of the fibers. We refer the reader to \cite{berline2003heat} for a construction.
\end{remark}

\subsubsection{Heat Operators and Smoothing Kernels}

We now recount from \cite{berline2003heat} the definition of the heat operator for the Bismut superconnection and the argument to show it has a smoothing kernel with an eye toward generalization in the next section. To do so, we first recall some facts about smoothing operators.

Let 
\begin{equation}\label{smoothing operators}
\mathcal K(\E) \raw S
\end{equation} 
be the bundle whose fiber over $z\in S$ is the space of smoothing operators on $\Gamma(X_z ; E_z)$. A section $K \in \mathcal K(\E)$ is a smooth family of vertical smoothing operators on $\E$ and has a smooth family of smooth integral kernels $k \in C^\infty(X \times_\pi X; \E \boxtimes \E^\ast)$. The smooth family $K$ acts on sections $\varphi \in \Gamma(\E)$ by 
   \begin{equation}
      (K\varphi)(x) := \int_{X/S}k(x,y)\varphi(y)d\text{Vol}_{X/S}(y)
   \end{equation}
where $d\text{Vol}_{X/S}$ is the vertical Riemannian volume measure for the vertical metric $g^{X/S}$.

The restriction of $K$ to the diagonal $X \subset X \times_\pi X$ is a section of $\End(\E)$. Applying the supertrace on $\End(\E)$ to this section gives a function on $X$ which can be integrated along the compact vertical fibers using the vertical volume form to give a smooth function on $S$. This is the \emph{supertrace} on sections of $\mathcal K(\E)$:

\begin{equation}\label{supertrace}
\begin{split}
   sTr: \Gamma(&\mathcal K(\E)) \raw C^\infty(S)\\
   K &\mapsto \int_{X/S} sTr \left( k(x,x)\right)d\text{Vol}_{X/S}(x). 
\end{split}
\end{equation}

For $t > 0$ the heat operator is defined using the \emph{Volterra series}
   \begin{equation}\label{Volterra}
   e^{-t\mathcal F} := e^{-t\slashed D^2} + \sum_{k\geq 1} (-t)^k I^{(k)}
   \end{equation}
where $e^{-t\slashed D^2} \in \Omega^0_S(\mathcal K(\E))$ is the heat operator for the Dirac Laplacian $\slashed D^2$ (see Remark \ref{dirac laplacian}) and 
   \begin{equation}I^{(k)} := \int_{\Delta_k} e^{-t\sigma_0\slashed D^2} \mathcal F^+ e^{-t\sigma_1 \slashed D^2} \cdots \mathcal F^+ e^{-t\sigma_k\slashed D^2}d\sigma
   \end{equation}
is a section of vertical operators obtained by integrating a continuous family of vertical operators over the $k$-simplex $\Delta_k := \{\sigma \in \R_+^{k+1} | \sigma_0 + \cdots + \sigma_k = 1\}$ with the standard Lebesgue measure.
Note that (\ref{Volterra}) is a finite sum since $I^{(k)}$ is a form of degree $\geq k$ and therefore vanishes for large $k$ due to the finite dimensionality of the exterior algebra $\Lambda(T^\ast S)$. 

The proof of the following lemma can be found in the appendix to Chapter 9 of \cite{berline2003heat}.

\begin{lemma}\label{smoothing lemma}
Let $D \in \Gamma(\mathcal D(\E))$ be a smooth family of vertical differential operators. Then the vertical operator
\begin{equation}
   \int_{\Delta_k} e^{-t\sigma_0\slashed D^2} D e^{-t\sigma_1\slashed D^2}\cdots De^{-t\sigma_k\slashed D^2}d\sigma
\end{equation}
is a family of vertical smoothing operators, i.e. it is a section of $\mathcal K(\E)$.
\end{lemma}

\begin{proposition}\label{smoothing}
The section of vertical operators $I^{(k)}$ is a form on $S$ valued in vertical smoothing operators, i.e.
$I^{(k)}\in \Omega{}_S(\mathcal K(\E))$
\end{proposition}

\begin{proof}
Since $\mathcal F^+$ is valued in sections of $\mathcal D(\E)$, applying Lemma \ref{smoothing lemma} to each term gives the desired result.
\end{proof}

\begin{corollary}\label{bismut smoothing}
For $t > 0$, the heat operator associated to the Bismut superconnection is a form valued in vertical smoothing operators, i.e. $e^{-t\mathcal F} \in \Omega_S(\mathcal K(\E))$
\end{corollary}

\subsubsection{Bismut's Theorem}
We now recall the ingredients necessary to state Bismut's index theorem for geometric families.

There is a one parameter family of geometric Clifford families obtained by scaling the vertical metric to $\frac{1}{t} g^{X/S}$ for each $t>0$. This gives rise to a one parameter family of Bismut superconnections
\begin{equation}\label{scaled superconnection}
   \mathbb B^t := t^{1/2}\slashed D + \left(\nabla^\mathscr \E + \frac{1}{2}k\right) -t^{1/2}\frac{1}{4}c(T)
\end{equation}
which can be expressed as
\begin{equation}
\mathbb B^t = t^{1/2}\delta_t \mathbb A \delta_t^{-1}
\end{equation}
if we set $\delta_t : \Omega_S(\sH) \raw\Omega_S(\sH)$ to be the operator that multiplies $\Omega^i_S(\sH)$ by $t^{-i/2}$.

The curvature of $\mathbb B^t$ is given by 
\begin{equation}
\mathcal F^t = (\mathbb B^t)^2 = \delta_t (t \mathcal F) \delta_t^{-1}.
\end{equation}
Using Corollary \ref{bismut smoothing} the \emph{Chern character} of $\mathbb B^t$ is defined as the supertrace

\begin{equation}\label{chern character}
\begin{split}
\ch(\mathbb B^t) &:= sTr(e^{-\mathcal F^t})\\
&= sTr(e^{-\delta_t t \mathcal F \delta_t^{-1}})\\
&= \delta_t sTr(e^{-t\mathcal F})\delta_t^{-1}
\end{split}
\end{equation}
where we have extended the supertrace (\ref{supertrace}) to a map $sTr: \Omega_S(\mathcal K(\E)) \raw \Omega_S$ on the level of forms.

Recall there is a map 
\begin{alignat}{3}
\alpha:\hspace{.25cm}& \so_n &&\raw \hspace{.25cm}& &\mathfrak{spin}_n \subset \Cl_n \nonumber\\
& A_{ij} &&\mapsto & & \frac{1}{4}\sum_{ij} A_{ij}c(e_i)c(e_j) \label{so rep}
\end{alignat}
The \emph{twisting curvature} $F^{\E/\mathbb S} \in \Omega^2_{X}(\End_{\Cl(T(X/S))}(\mathscr E))$ is

\begin{equation}\label{twisting curvature}
   F^{\mathscr E/\mathbb S} := (\nabla^{\mathscr E})^2 - c(\alpha(\Omega^{X/S}))
\end{equation}
Here $c: \Cl(T(X/S)) \raw \End(\E)$ is (\ref{clifford map}) and $\Omega^{X/S}$ is (\ref{vertical curvature}). We refer the reader to Proposition 3.43 in \cite{berline2003heat} for a justification of the statement that $F^{\E/\mathbb S}$ is valued in endomorphisms that commute with the vertical Clifford action.
In particular, $F^{\mathscr E/\mathbb S}$ commutes with $\Gamma$ defined in (\ref{chirality operator}) and therefore preserves the $\Z/2$-grading on $\E$.

The \emph{relative Chern character} of $F^{\E/\mathbb S}$ is then defined to be
\begin{equation}\label{relative chern}
\ch_{\E/\mathbb S}(F^{\E/\mathbb S}) := 2^{-n/2}sTr(\exp(-\Gamma \cdot F^{\mathscr E/\mathbb S})
\end{equation}
 where $n$ is the dimension of the even dimensional fiber of $\pi$. 

We can now state the families index theorem:
\begin{theorem}[\cite{Bismut:1986uh}]\label{family index theorem}
   \begin{equation}
      \lim_{t\raw 0} \ch(\mathbb B^t) = \int_{X/S} \hat A(\Omega^{X/S})\ch_{\E/\mathbb S}(F^{\E/\mathbb S})
   \end{equation}
\end{theorem}

\section{The Equivariant Families Index Theorem}\label{Section 4}
We now embark on proving an equivariant version of Theorem \ref{family index theorem} using the tools laid out in \S \ref{Section 2}.

Let $(\E \raw X \xrightarrow{\pi} S, g^{X/S}, H ,\nabla^\E)$ be a geometric Clifford family with a $G$-action that preserves the geometric data.  
There is an iterated bundle of Borel quotients 
\begin{equation}\label{borel}
(\mathscr E_G)_\nabla \raw (X_G)_\nabla \raw (S_G)_\nabla \raw B_\nabla G
\end{equation}
and pulling back (\ref{borel}) along the map (\ref{classifying map}) gives an iterated map of fiber bundles
\begin{equation}
\E_Q \raw X_Q \xrightarrow{\pi_Q} S_Q \raw M
\end{equation} 
associated to the principal $G$-bundle $Q\raw M$ with connection $\Phi$. 

We recall that the total space of the bundle $\mathscr H \raw S$ is infinite dimensional and we regard it as a sheaf of sets on the site of manifolds. In particular, a map from the test manifold $M$ to $\mathscr H$ consists of a map $\varphi: M \raw S$ and a section of the pullback bundle $\varphi^\ast \mathscr H \raw M$, which can be identified as a section of $\varphi^\ast \E \raw \varphi^\ast X$. Thus the set $\mathscr H(M)$ fibers over $S(M)$ where the fiber over $\varphi \in S(M)$ is $\Gamma(\varphi^\ast X; \varphi^\ast \E)$.

The action on sections of $\E$ induces an action on the bundle $\mathscr H$. We use this action to construct the simplicial sheaf $(\sH_G)_\nabla$ whose value on $M$ is the simplicial set
\begin{equation}
(\sH_G)_\nabla(M) := E_\nabla G(M) \times \sH(M) \leftleftarrows E_\nabla G(M) \times \sH(M) \times G(M) \mathrel{\substack{\textstyle\leftarrow\\[-0.6ex]
                      \textstyle\leftarrow \\[-0.6ex]
                      \textstyle\leftarrow}} \cdots
\end{equation}

This fits into a bundle of simplicial Borel quotients 
\begin{equation}
(\sH_G)_\nabla \raw (S_G)_\nabla \raw B_\nabla G
\end{equation}
that pulls back along $f$ to 
\begin{equation}
\sH_Q \raw S_Q \raw M.
\end{equation}

There is a horizontal distribution $H_Q:= q_\ast(V \oplus H)$ on $\pi_Q$  obtained by applying Construction \ref{distribution} to $\pi$, where $V$ is (\ref{V distribution}) .  By $G$-invariance, the metric $p^\ast g^{X/S}$ descends to a metric $g^{X/S}_Q$ on $T(X_Q/S_Q)$. Thus we have shown

\begin{lemma}\label{geometric family lemma}
$(X_Q \xrightarrow{\pi_Q} S_Q, H_Q, g^{X/S}_Q)$ is a geometric family.
\end{lemma}

Applying Lemma \ref{Q covariant derivative} to $\nabla^\E$ implies there is a covariant derivative $\nabla^{\mathscr E}_Q$ on $\E_Q \raw X_Q$. Similarly there is a covariant derivative $\nabla^{X/S}_Q$ on $T(X_Q/S_Q)$. In fact, there is yet another covariant derivative $\nabla^{X_Q/S_Q}$ on $T(X_Q/S_Q)$ obtained by applying Construction \ref{vertical connection} to the geometric family and they are equal:
\begin{equation}
\nabla^{X_Q/S_Q} = \nabla^{X/S}_Q.
\end{equation}
We refer the reader to Lemma 16 of \cite{EquivDet} for a proof of this statement.

\begin{remark}
The curvature 
\begin{equation}\label{vertical Q curvature}
\Omega_Q^{X/S} := (\nabla^{X/S}_Q)^2
\end{equation}
is a form in $\Omega_{X_Q}(\so(TX/S)) \simeq \Omega_{\mathcal B_{SO}(T(X_Q/S_Q))}(\so(n))$. It is identified with (\ref{Q curvature}) when $P = \mathcal B_{SO}(T(X/S))$ is the bundle of frames on $T(X/S)$.
\end{remark}

\begin{proposition}
$(\E_Q \raw X_Q \raw S_Q, g^{X/S}_Q, H_Q, \nabla^\E_Q)$ is a geometric Clifford family.
\end{proposition}

\begin{proof}
   In light of Lemma \ref{geometric family lemma}, we are left with showing that $\E_Q \raw X_Q$ is a vertical Clifford module bundle with Clifford connection $\nabla^\E_Q$.
   Clifford multiplication on $\E_Q$ by $\Cl(T(X_Q/S_Q), g^{X/S}_Q)$ is defined fiberwise as the action of $\Cl(T(X/S), g^{X/S})$ on $\E$, i.e. sections $v \in \Gamma(X_Q; T(X_Q/S_Q))$, $\sigma \in \Gamma(X_Q, \E_Q)$ lift to $G$-equivariant sections $\tilde v:= q^\ast v \in \Gamma(Q\times X; \underline{T(X/S)})^G$, $\tilde \sigma:= q^\ast \sigma \in \Gamma(Q\times X; \underline \E)^G$ and we set
   \begin{equation}
      v\cdot\sigma := q_\ast(\tilde v\cdot \tilde \sigma).
   \end{equation}

   Recall that by Lemma \ref{Q covariant derivative} 
   \begin{equation}
   q^\ast \nabla^\E_Q = \nabla^\E - \phi^a \tau^\E_a
   \end{equation}
   where $\tau^\E_a := \mathcal L_a - \nabla^\E_a$ in particular has the property
   \begin{equation}
   \begin{split}
      \tau^\E_a(\tilde v \cdot \tilde \sigma) 
      &= (\mathcal L_a - \nabla^\E_a)(\tilde v \cdot \tilde \sigma)\\
      &= -\nabla^\E_a(\tilde v \cdot \tilde \sigma)\\
      &= -\nabla^{X/S}_a \tilde v \cdot \tilde \sigma - \tilde v\cdot \nabla^\E_a \tilde \sigma\\
      &= \tau^{X/S}_a(\tilde v)\cdot \tilde \sigma + \tilde v \cdot \tau^\E_a(\tilde \sigma)
   \end{split}
   \end{equation}
   where we used the $G$-equivariance of $\tilde v, \tilde \sigma$ in the second and fourth equalities.

   We now check
   \begin{equation}
   \begin{split}
      \nabla^\E_Q(v\cdot\sigma) 
      &= q_*\left(\nabla^\E(\tilde v \cdot \tilde \sigma) - \phi^a\tau^\E_a(\tilde v \cdot \tilde \sigma)\right)\\
      &= q_*\left(\nabla^{X/S}\tilde v \cdot \tilde \sigma + \tilde v \cdot \nabla^\E \tilde \sigma - \phi^a (\tau^{X/S}_a\tilde v) \cdot \tilde \sigma - \phi^a \tilde v \cdot \tau^\E_a\tilde \sigma\right)\\
      &= q_*\left(\left(\nabla^{X/S}\tilde v - \phi^a\tau^{X/S}_a\tilde v\right)\cdot \tilde \sigma + \tilde v \cdot \left(\nabla^\E\tilde \sigma - \phi^a \tau^\E_a\tilde \sigma\right)\right)\\
      &= (\nabla_Q^{X/S}v)\cdot \sigma + v \cdot \nabla^\E_Q \sigma
   \end{split}
   \end{equation}
\end{proof}

As before, there is a 1-parameter family of Bismut superconnections 
\begin{equation}\label{Q bismut}
 \mathbb B^t_Q: \Omega_{S_Q}(\sH_Q) \raw \Omega_{S_Q}(\sH_Q)
\end{equation}
associated to the 1-parameter family of geometric Clifford families $(\E_Q \raw X_Q \raw S_Q, \frac{1}{t}g^{X/S}_Q,H_Q, \nabla^\E_Q)$. 
For $t>0$, $\mathbb B^t_Q$ is an odd derivation over the algebra of forms $\Omega_{S_Q}$defined on $\Omega_{S_Q}(\sH_Q)$ by

\begin{equation}
\mathbb B^t_Q := t^{1/2}\slashed D_Q + \left(\tilde\nabla^{\E}_Q + \frac{1}{2}k_Q\right) -t^{-1/2}\frac{1}{4}c(T_Q).
\end{equation}
Here $\slashed D_Q \in \Omega^0_{S_Q}(\mathcal D(\E_Q))$ is the vertical Dirac operator, $\tilde \nabla^\E_Q$ is the covariant derivative on $\sH_Q$ constructed from $\nabla^\E_Q$ as in (\ref{tilde nabla}), $k_Q \in \Omega^1_{S_Q}(\mathcal D(\E_Q))$ is multiplication by the mean curvature tensor (\ref{mean curvature}), and $c(T_Q) \in \Omega^2_{S_Q}(\mathcal D(\E_Q))$ is Clifford multiplication by the Frobenius tensor (\ref{frobenius}). 
The curvature
\begin{equation}
\mathcal F_Q^t : = (\mathbb B^t_Q)^2
\end{equation}
is a form in $\Omega_{S_Q}(\mathcal D(\E_Q))$.

\begin{lemma}
The vertical Dirac operator satisfies
\begin{equation}
\slashed D_Q = \slashed D 
\end{equation}
in $\Omega^0_{Q\times X}(\mathcal D(\underline \E))_{\basic}$.
\end{lemma}

\begin{proof}
The vertical Dirac operator $\slashed D \in \Omega^0_X(\mathcal D(\E))$ is $G$-invariant because the action preserves the vertical metric $g^{X/S}$. This implies $p^\ast \slashed D \in \Omega^0_{Q\times X}(\mathcal D(\underline \E))$ is also $G$-invariant and it descends to $\slashed D_Q$.
\end{proof}

There is a projection \cite{MATHAI198685}, which we call the \emph{Mathai-Quillen map},
\begin{equation}\label{Mathai Quillen}
\exp(-\phi^a\iota_a): \Omega_{Q\times X} \raw (\Omega_{Q\times X})^{hor}
\end{equation}
onto horizontal forms on $Q\times X$, i.e. those that are annihilated by $\iota_a$ for all $e_a \in \g$. It restricts to a map on $G$-invariants
\begin{equation}
\Omega_X^G \xhookrightarrow{p^\ast} \Omega_{Q\times X}^G \xrightarrow{\exp(-\phi^a\iota_a)} (\Omega_{Q\times X})_{\bas}.
\end{equation}
We note that the vertical volume form for the geometric family $(\pi, g^{X/S}, H)$ is $G$-invariant and we identify it with a form $\nu^{X/S} \in (\Omega^n_X)^G$ using the horizontal distribution $H$.

\begin{lemma}\label{Q vert volume}
The image of $\nu^{X/S}$ under the Mathai-Quillen map is the vertical volume form $\nu^{X_Q/S_Q} \in (\Omega^n_{Q\times X})_{\bas}$ on the geometric family $(\pi_Q, g^{X/S}_Q, H_Q)$, i.e. 
\begin{equation}
\begin{split}
   \nu^{X_Q/S_Q} &= \exp(-\phi^a\iota_a)\nu^{X/S}\\ 
   &= \nu^{X/S} - \phi^a\iota_a\nu^{X/S} + \frac{1}{2}\phi^a\phi^b\iota_b\iota_a \nu^{X/S} - \cdots
\end{split}
\end{equation}
\end{lemma}

\begin{proof}
The restriction to a fiber of the form $\exp(-\phi^a\iota_a)\nu^{X/S}$ agrees with the vertical Riemannian volume form. Moreover, for any $\xi \in V\oplus H$, $\iota_\xi \exp(-\phi^a\iota_a)\nu^{X/S} =0$. Since these two properties uniquely characterize $\nu^{X_Q/S_Q}$, this concludes the proof.
\end{proof}

\begin{proposition}
The mean curvature satisfies
\begin{equation}\label{result}
k_Q = k - \phi^a\iota_a k
\end{equation}
in $\Omega^1_{Q\times S}(\mathcal D(\underline \E))_{\basic}$.
\end{proposition}

\begin{proof}
For the purposes of this proof only, we will call a form $\omega \in \Omega_{Q\times X}$ \emph{irrelevant} if 
\begin{equation}
\iota_{e_1}\cdots\iota_{e_n} \omega = 0
\end{equation}
for $(e_i)$ an orthonormal frame of the vertical tangent bundle, i.e. $\omega$ is irrelevant if it does not contain a factor of $\nu^{X/S}$.

We now apply the previous lemma. The equation characterizing $k_Q$ is
\begin{equation}\label{dnu}
\begin{split}
   d\nu^{X_Q/S_Q} &
   = k_Q \wedge \nu^{X_Q/S_Q} + \sum_{i =1}^n \langle T_Q, e_i\rangle \iota_{e_i}\nu^{X_Q/S_Q}\\
   &= k_Q \wedge \nu^{X/S} + \text{irrelevant terms}
\end{split}
\end{equation}

The left hand side of (\ref{dnu}) is
\begin{equation}
\begin{split}
d\nu^{X_Q/S_Q} &= d\nu^{X/S} + \phi^a d\iota_a \nu^{X/S} + \text{irrelevant terms}\\
&= k\wedge \nu^{X/S} - \phi^a\iota_a d\nu^{X/S} + \text{irrelevant terms}\\
&= (k - \phi^a\iota_a k) \wedge \nu^{X/S} + \text{irrelevant terms}
\end{split}
\end{equation}
where we used $G$-invariance of $\nu^{X/S}$ to say that $\mathcal L_a \nu^{X/S} = [d, \iota_a]\nu^{X/S} = 0$ in the second equality. Matching the relevant terms gives
\begin{equation}\label{almost}
k_Q = k - \phi^a\iota_a k
\end{equation}
as forms in $(\Omega_{Q\times X})_{\bas}$. Let $(\Omega^{H_Q}_{Q\times X})_{\bas}$ denote those forms annihilated by $\iota_\zeta$ where $\zeta$ is a section of $T(X/S)$ and let $\pi_\ast \underline \R$ denote the bundle on $Q\times S$ whose preimage over a point in $Q\times S$ is the space of smooth functions on the fiber. Under the identification 
\begin{equation}\label{final identification}
(\Omega^{H_Q}_{Q\times X})_{\bas} \simeq \Omega_{Q\times S}(\pi_\ast \underline \R)_{\bas} \subset \Omega_{Q\times S}(\mathcal D(\underline\E))_{\bas}
\end{equation}
(\ref{almost}) becomes (\ref{result}).
\end{proof}

\begin{proposition}
The Frobenius tensor satisfies
\begin{equation}\label{frobenius relation}
T_Q = T - \phi^a \iota_a T+ \frac{1}{2}\phi^a \phi^b \iota_b\iota_a T - \omega^a P\xi_a
\end{equation}
where $P$ is (\ref{split exact}) and $\xi_a$ is the vector field on $X$ generated by the action of $e_a \in \g$.
\end{proposition}

\begin{proof}
We first set some notation. Given a split short exact sequence of vector bundles
\begin{equation}
0 \raw W \mono E \stackrel[\iota]{j}{\rightleftarrows} V \raw 0
\end{equation}
we denote by $P^V_W:= \iota \circ j \in \End(E)$ the projection operator onto $\iota(V)$ with kernel $W$.

By definition,
\begin{equation}
T_Q := P^{(T(X_Q/S_Q)}_{H_Q}[P^{H_Q}_{T(X_Q/S_Q)} -, P^{H_Q}_{T(X_Q/S_Q)}-]
\end{equation}
which is a tensor, i.e. an element of $\Omega^2_{X_Q}(T(X_Q/S_Q))$. To show (\ref{frobenius relation}), we will evaluate $q^\ast T_Q \in \Omega^2_{Q\times X}(\underline{T(X/S)})_{\bas}$ pointwise on vectors $\eta, \zeta \in T(Q\times X)$. By tensoriality, this value does not change regardless of how we extend $\eta, \zeta$ to global vector fields and we will make use of this freedom when doing the calculation.

Now we observe that pointwise
\begin{equation}\label{commute projector 1}
P^{H_Q}_{T(X_Q/S_Q)} \circ q_\ast = q _\ast \circ P^{V\oplus H}_{\dot\rho(\g) \oplus T(X/S)}
\end{equation}
and
\begin{equation}\label{commute projector 2}
P^{T(X_Q/S_Q)}_{H_Q} \circ q_\ast = q _\ast \circ P_{\dot\rho(\g) \oplus V\oplus H}^{T(X/S)}.
\end{equation}
Here $\dot\rho(\g) \subset T(Q\times X)$ is the subbundle generated the action of $\g$ on $Q\times X$. We will use $\dot\rho_X: \g \raw C^\infty(TX)$ and $\dot\rho_Q: \g \raw C^\infty(TQ)$ to denote the action of $\g$ on $X$ and $Q$.

We now use our freedom and extend $\eta, \zeta$ to $G$-invariant vector fields on $Q\times X$. Then
\begin{equation}\label{calculation}
\begin{split}
(q^*T_Q)(\eta, \zeta) &= q^* P^{T(X_Q/S_Q)}_{H_Q}[P^{H_Q}_{T(X_Q/S_Q)} q_\ast \eta, P^{H_Q}_{T(X_Q/S_Q)}q_\ast \zeta]\\
&= q^*P^{T(X_Q/S_Q)}_{H_Q}q_*[P^{V\oplus H}_{\dot\rho(\g) \oplus T(X/S)}\eta, P^{V\oplus H}_{\dot\rho(\g) \oplus T(X/S)}\zeta]\\
&= P_{\dot\rho(\g) \oplus V\oplus H}^{T(X/S)}[P^{V\oplus H}_{\dot\rho(\g) \oplus T(X/S)}\eta, P^{V\oplus H}_{\dot\rho(\g) \oplus T(X/S)}\zeta]
\end{split}
\end{equation}
where (\ref{commute projector 1}) and (\ref{commute projector 2}) were used in the second and third equalities.

We note that $\eta$ has a decomposition
\begin{equation}
\eta = \eta_Q + \eta_X
\end{equation}
into components in $TQ\oplus TX$, which has a further decomposition
\begin{equation}
\eta = \underbrace{\left(\dot\rho_Q\Phi(\eta) + \dot\rho_X\Phi(\eta)\right)}_{\in \dot\rho(\g)} 
+ \underbrace{\left(\eta_Q - \dot\rho_Q\Phi(\eta)\right)}_{\eta_V \in V} 
+ \underbrace{P^H_{T(X/S)}\left(\eta_X - \dot\rho_X\Phi(\eta)\right)}_{\eta_H \in H} 
+ \underbrace{P^{T(X/S)}_H\left(\eta_X - \dot\rho_X\Phi(\eta)\right)}_{\in T(X/S)}.
\end{equation}
There is a similar decomposition for $\zeta$. 

Continuing (\ref{calculation}),
\begin{equation}
\begin{split}
(q^*T_Q)(\eta,\zeta) &= P_{\dot\rho(\g) \oplus V\oplus H}^{T(X/S)}[\eta_H + \eta_V, \zeta_H + \zeta_V]\\
&= P_{\dot\rho(\g) \oplus V\oplus H}^{T(X/S)}
\left([\eta_H, \zeta_H] 
+ \underbrace{[\eta_H, \zeta_V] + [\eta_V, \zeta_H]}_{\in V\oplus H} 
+ [\eta_V, \zeta_V]\right)\\
&= \underbrace{P^{T(X/S)}_H[P^H_{T(X/S)} \eta_X, P^H_{T(X/S)} \zeta_X ]}_{\iota_\zeta\iota_\eta T}\\
&\hspace{.4cm} - \underbrace{\left(P^{T(X/S)}_H[P^H_{T(X/S)}\dot\rho_X\Phi(\eta), P^H_{T(X/S)}\zeta_X]
+ P^{T(X/S)}_H[ P^H_{T(X/S)}\eta_X, P^H_{T(X/S)}\dot\rho_X\Phi(\zeta)]\right)}_{\iota_\zeta\iota_\eta \phi^a\iota_a T}\\
&\hspace{.4cm} + \underbrace{P^{T(X/S)}_H[ P^H_{T(X/S)}\dot\rho_X\Phi(\eta), P^H_{T(X/S)}\dot\rho_X\Phi(\zeta)]}_{\frac{1}{2}\iota_\xi\iota_\eta\phi^a\phi^b\iota_b\iota_a T}\\
&\hspace{.4cm} - \underbrace{P^{T(X/S)}_H\dot\rho_X\omega(\eta, \zeta)}_{\iota_\zeta\iota_\eta\omega^a P \xi_a}\\
\end{split}
\end{equation}
\end{proof}
Applying the Clifford map (\ref{clifford map}) and the identification (\ref{final identification}), we get
\begin{corollary} There is an equality
\begin{equation}
c(T_Q) = T - \phi^a \iota_a c(T)+ \frac{1}{2}\phi^a \phi^b \iota_b\iota_a c(T) - \omega^a c(P\xi_a)
\end{equation}
of forms in $\Omega_{Q\times S}(\mathcal D(\underline \E))_{\bas}$
\end{corollary}

\subsubsection{The Equivariant Bismut Superconnection}

We define the equivariant Bismut superconnection
to be the derivation 
\begin{equation}
\mathbb B_G: \Omega_{(S_G)_\nabla}((\sH_G)_\nabla)\raw \Omega_{(S_G)_\nabla}((\sH_G)_\nabla)
\end{equation} over the algebra $\Omega_{(S_G)_\nabla} \simeq \left(\overline\Kos(\g^\ast)\otimes \Omega_S\right)_{\bas}$ given by
\begin{equation}
\mathbb B_G := \slashed D + \left(\tilde \nabla^\E_G + \frac{1}{2}k_G\right) - \frac{1}{4}c(T)_G
\end{equation}
where the terms are given as follows:
\begin{itemize}
 \item $\slashed D \in \Omega^0_S(\mathcal D(\E))$ is the vertical Dirac operator. It is $G$-equivariant and defines a degree 0 map on $\Omega_{(S_G)_\nabla}((\sH_G)_\nabla) := (\Kos(\g^*)\otimes \Omega_S(\sH))_{\bas}$.
 \item $\tilde \nabla^\E_G$ is defined as
\begin{equation}
\tilde \nabla^\E_G := d_K + \tilde\nabla^\E - \theta^a\tau_a
\end{equation}
where $\tilde \nabla^\E$ is (\ref{tilde nabla}) and $\tau_a$ is (\ref{endomorphism})
\item $k_G \in (\Kos(\g^*) \otimes \Omega_S(\mathcal D(\E)))^{\bullet = 1}_{\bas}$ is
\begin{equation}
k_G := k - \theta^a\iota_a k 
\end{equation}
\item $c(T)_G \in (\Kos(\g^*) \otimes \Omega_S(\mathcal D(\E)))_{\bas}^{\bullet =2}$ is
\begin{equation}
c(T)_G := c(T) - \theta^a \iota_a c(T) + \frac{1}{2}\theta^a\theta^b\iota_b\iota_a c(T) - \chi^a c(P(\xi_a))
\end{equation}
where $\chi^a$ is (\ref{chi}).
\end{itemize}

\begin{proposition}\label{pullback supercurvature}
The curvature 
\begin{equation}
\mathcal F_G := (\mathbb B_G)^2
\end{equation}
is a form in $(\Kos(\g^\ast)\otimes \Omega_S(\mathcal D(\E)))_{\basic}$ and it satisfies
\begin{equation}
f_S^\ast \mathcal F_G = \mathcal F_Q
\end{equation}
\end{proposition}

\begin{proof}
The map 
\begin{equation}
\mathbb B_G^2: (\Kos(\g^*) \otimes \Omega_S(\sH))_{\bas} \raw (\Kos(\g^*) \otimes \Omega_S(\sH))_{\bas}
\end{equation}
supercommutes with exterior multiplication by forms in $\Omega_{(S_G)_\nabla}$, and therefore defines an element of $(\Kos(\g^\ast)\otimes \Omega_S(\mathcal D(\E)))_{\basic}$.

Let $\overline{\mathbb B}_Q: \Omega_{Q\times S}(\underline \sH) \raw \Omega_{Q\times S}(\underline \sH)$ be the extension of $\mathbb B_Q$  using  formula (\ref{Q bismut}) with $t=1$. Similarly, let $\overline{\mathbb B}_G$ be the extension of $\mathbb B_G$ to $\overline \Kos(\g^*) \otimes \Omega_S(\sH)$. Then $\overline{\mathbb B}_G^2: \overline \Kos(\g^*)\otimes \Omega_S(\sH) \raw \overline \Kos(\g^*)\otimes \Omega_S(\sH)$. One checks that

\begin{equation}
\begin{tikzcd}
   \overline \Kos(\g^\ast) \otimes \Omega_S(\sH) \arrow[r, "\overline{\mathbb B}_G"] \arrow[d, "f_S^\ast"] & \overline\Kos(\g^\ast) \otimes \Omega_S(\sH) \arrow[d, "f_S^\ast"]\\
   \Omega_{Q \times S}(\underline \sH) \arrow[r, "\overline{\mathbb B}_Q"] & \Omega_{Q\times S}(\underline \sH)
\end{tikzcd}
\end{equation}
commutes, which implies that
\begin{equation}
\begin{tikzcd}
   \Omega^0_S(\sH) \arrow[r, "\overline{\mathbb B}_G^2"] \arrow[d, "f_S^\ast "] & \overline \Kos(\g^\ast) \otimes \Omega_X(E) \arrow[d, "f_S^*"]\\
   \Omega^0_{Q\times S}(\underline \sH) \arrow[r, "\overline{\mathbb B}_Q^2"] & \Omega_{Q\times S}(\underline \sH)
\end{tikzcd}
\end{equation}
commutes.

Since sections of $\underline \sH \simeq p^\ast \sH$ are generated over the ring of functions on $Q\times S$ by $p^\ast\sigma$ for $\sigma \in \Omega^0_S(\sH)$, we see that $\overline{\mathbb B}_Q^2 = f_S^\ast \overline{\mathbb B}_G^2$ which implies $\mathcal F_Q = f_S^\ast \mathcal F_G$.
\end{proof}

We can write
\begin{equation}
\mathcal F_G = \slashed D^2 + \mathcal F^+_G
\end{equation}
where $\mathcal F^+_G \in (\overline \Kos(\g^*) \otimes \mathcal D(\E))_{\bas}$ is a form with terms of degree between 1 and 4,  valued in vertical differential operators of order 1.

As in (\ref{Volterra}) we define the heat operator of the equivariant Bismut superconnection as the Volterra series

\begin{equation}\label{Volterra2}
   e^{-t\mathcal F_G} := e^{-t\slashed D^2} + \sum_{k\geq 1} (-t)^k I^{(k)}_G
\end{equation}
where 
\begin{equation}I^{(k)}_G := \int_{\Delta_k} e^{-t\sigma_0\slashed D^2} \mathcal F^+_G e^{-t\sigma_1 \slashed D^2}\cdots \mathcal F^+_G e^{-t\sigma_k \slashed D^2}d\sigma.
\end{equation} 

A notable difference to (\ref{Volterra}) is that (\ref{Volterra2}) defines a form of unbounded degree on $(S_G)_\nabla$ with coefficients in vertical operators on $\E$. However, we have the following analogous statements

\begin{proposition}
The operators $I_G^{(k)}$ are forms of degree $\geq k$ on $(S_G)_\nabla$ valued in vertical smoothing operators, i.e. $I^{(k)}_G \in (\overline \Kos(\g^*) \otimes \Omega_S(\mathcal K(\E)))_{\bas}$.
\end{proposition}

\begin{proof}
The form $\mathcal F^+_G$ is valued in sections of $\mathcal D(\E)$ so applying Lemma (\ref{smoothing lemma}) to each term shows that $I^{(k)}_G$ is valued in sections of $\mathcal K(\E)$. It is clear that $I^{(k)}_G$ is horizontal and $G$-invariant. 
\end{proof}

\begin{corollary}
For $t>0$, the heat operator is an unbounded form valued in vertical smoothing operators, i.e. $e^{-t\mathcal F_G} \in (\overline \Kos(\g^*) \otimes \Omega_S(\mathcal K(\E)))_{\bas}$.
\end{corollary}

\begin{lemma}\label{equivariant supertrace}
The supertrace
\begin{equation}
   sTr: \overline \Kos(\g^*) \otimes \Omega_S(\mathcal K(\E)) \raw \overline \Kos(\g^*) \otimes \Omega_S
\end{equation}
is $G$-equivariant.
\end{lemma}

\begin{proof}
For $K \in \mathcal K(\E)$, $\varphi \in C^\infty(\E)$, and $g \in G$ we have
\begin{equation}
\begin{split}
   (g.K)(\varphi) &= gK(g^{-1}.\varphi)\\
   &= g \int_{X/S}k(x,y)\varphi(g.y)\\
   &= \int_{X/S}k(g^{-1}x, y)\varphi(g.y)d\text{Vol}_{X/S}(y)\\
   &= \int_{X/S}k(g^{-1}x, g^{-1}y)\varphi(y) d\text{Vol}_{X/S}(y)
\end{split}
\end{equation}
where we have used the $G$-invariance of $g^{X/S}$ in the fourth equality. Thus
\begin{equation}
\begin{split}
   sTr(g.K)(z) &= \int_{X_z/S_z} sTr \left(k(g^{-1}x, g^{-1}x)\right) d\text{Vol}_{X/S}(x)\\
   &= sTr(K)(g^{-1}.z) = g.sTr(K)(z)
\end{split}
\end{equation}
\end{proof}
Lemma \ref{equivariant supertrace} implies that the map
\begin{equation}
   1\otimes sTr: \overline \Kos(\g^*) \otimes \Omega_S(\mathcal K(\E)) \raw \overline \Kos(\g^*)\otimes \Omega_S
\end{equation}
is $G$-equivariant and preserves horizontal forms so it restricts to a map on basic forms
\begin{equation}
   sTr: (\overline \Kos(\g^\ast) \otimes \Omega_S(\mathcal K(\E)))_{\bas} \raw (\overline \Kos(\g^\ast) \otimes \Omega_S)_{\bas}.
\end{equation}

Analogous to (\ref{scaled superconnection}), for $t > 0$ the scaled equivariant superconnection satisfies
\begin{equation}
\begin{split}
\mathbb B^t_G &:= t^{1/2}\slashed D + \left(\tilde \nabla^\E_G + \frac{1}{2}k_G\right) - t^{-1/2}\frac{1}{4}c(T)_G\\
&= t^{1/2}\delta_t\circ \mathbb B_G \circ \delta_t^{-1}.
\end{split}
\end{equation}

We now define the Chern character of the scaled equivariant superconnection to be 
\begin{equation}
\begin{split}
\ch(\mathbb B^t_G) &:= sTr(e^{-\mathcal F^t_G})\\
&= sTr(e^{-\delta_t t\mathcal F_G \delta_t^{-1}})\\
&= \delta_t sTr(e^{-t\mathcal F_G})\delta_t^{-1}
\end{split}
\end{equation}
which is a form in $\Omega_{(S_G)_\nabla}$.

\begin{proposition}\label{pullback chern character}
The Chern Weil map (\ref{pullback}) applied to the Chern character is a form in $\Omega_{S_Q}$ that satisfies 
\begin{equation}
f_S^\ast \ch(\mathbb B^t_G) = \ch(\mathbb B^t_Q)
\end{equation}
\end{proposition}

\begin{proof}
It follows from \ref{pullback supercurvature} that
\begin{equation}
f_S^* e^{-t\mathcal F_G} = e^{-t\mathcal F_Q}
\end{equation}
and a check of the definitions shows that
\begin{equation}
\begin{tikzcd}
(\overline \Kos(\g^*) \otimes \Omega_S(\mathcal K(\E)))_{\bas} \arrow[r, "sTr"] \arrow[d, "f_S^\ast"] &
(\overline \Kos(\g^*) \otimes \Omega_S)_{\bas} \arrow[d, "f_S^\ast"]\\
\Omega_{S_Q}(\mathcal K(\E_Q)) \arrow[r, "sTr"] & \Omega_{S_Q}
\end{tikzcd}
\end{equation}
commutes.

\end{proof}

Let $\nabla^{X/S}_G$ be the covariant derivative on $(\overline \Kos(\g^*) \otimes \Omega_X(T(X/S)))_{\bas}$ given by
\begin{equation}
\nabla^{X/S}_G := d_K + \nabla^{X/S} - \theta^a\tau_a.
\end{equation}
Its curvature
\begin{equation}
   \Omega^{X/S}_G := (\nabla^{X/S}_G)^2
\end{equation}
is a form in $(\Kos(\g^*)\otimes \Omega_X(\so(T(X/S))))_{\bas}$.
By Lemma \ref{pullback equivariant curvature}, it satisfies
\begin{equation}\label{pullback vertical curvature}
f_X^\ast \Omega^{X/S}_G = \Omega^{X/S}_Q
\end{equation}
where $\Omega^{X/S}_Q$ is (\ref{vertical Q curvature}). 
The $\hat A$ polynomial applied to the curvature $\Omega^{X/S}_G$ is the unbounded form
\begin{equation}
\hat A(\Omega^{X/S}_G) := \text{det}^{1/2}\left(\frac{\Omega^{X/S}_G/2}{\sinh(\Omega^{X/S}_G/2)}\right)
\end{equation}
lying in $(\overline \Kos(\g^\ast) \otimes \Omega_X(\so(T(X/S))))_{\bas}$.
Applying (\ref{pullback vertical curvature}) gives
\begin{equation}
\hat A(\Omega_Q^{X/S}) = f_X^\ast \hat A(\Omega_G^{X/S}).
\end{equation}

Finally, let $F^{\E}_G \in (\Kos(\g^\ast)\otimes \Omega_X(\End(\E)))_{\basic}$ be the curvature of $\nabla^\E_G$ and set 
\begin{equation}
F^{\E/\mathbb S}_G := F_G^{\E} - c(\alpha(\Omega_G^{X/S}))
\end{equation} in $(\Kos(\g^\ast)\otimes \Omega_X(\End(\E)))_{\basic}$, where $\alpha: \so(n) \raw \mathfrak{spin}_n \subset \Cl_n$ is (\ref{so rep}) and $c: \Cl(T(X/S)) \raw \End(\E)$ is the Clifford map.
Again, Lemma \ref{pullback equivariant curvature} implies that
\begin{equation}
F^{\E/\mathbb S}_Q = f_X^\ast F^{\E/\mathbb S}_G
\end{equation}
and applying the relative chern character (\ref{relative chern}) gives the relation
\begin{equation}
\ch_{\E/\mathbb S}(F^{\E/\mathbb S}_Q) = f_X^\ast \ch_{\E/\mathbb S}(F^{\E/\mathbb S}_G).
\end{equation}

We now prove the main theorem.
\begin{theorem}\label{main theorem}
There is an equality
\begin{equation}
   \lim_{t\raw 0}\ch(\mathbb B^t_G) = \int_{X/S} \hat A(\Omega^{X/S}_G) \ch_{\E/\mathbb S}(F^{\E/\mathbb S}_G).
\end{equation}
of unbounded forms in $\overline \Omega_{(S_G)_\nabla}$.
\end{theorem}

\begin{proof}
The statement is proved if we show that the equation is satisfied on pullbacks along maps $f: M \raw B_\nabla G$:
\begin{equation}
\begin{split}
\lim_{t\raw 0} f_S^\ast \ch(\mathbb B^t_G) 
&= \lim_{t\raw 0} \ch(\mathbb B^t_Q)\\
&= \int_{X_Q/S_Q} \hat A(\Omega_Q^{X/S})\ch_{\E/\mathbb S}(F_Q^{\E/\mathbb S})\\
&= \int_{X_Q/S_Q} f_X^\ast \hat A(\Omega_G^{X/S}) f_X^\ast \ch_{\E/\mathbb S}(F^{\E/\mathbb S}_G)\\
&= f_S^\ast \int_{X/S} \hat A(\Omega^{X/S}_G) \ch_{\E/\mathbb S}(F^{\E/\mathbb S}_G)
\end{split}
\end{equation}
where we used Proposition \ref{pullback chern character} in the first equality and Theorem \ref{family index theorem} in the second. The fourth equality follows from the equality 
\begin{equation}
\int_{(Q\times X)/(Q\times S)}p^\ast \omega = p^\ast \int_{X/S} \omega
\end{equation}
 for all $\omega \in \Omega_X$.
\end{proof}

\bibliographystyle{alpha}
\bibliography{bibliography}

\begin{thebibliography}{BGV03}

\bibitem[BF86]{Bismut:1986wl}
Jean-Michel Bismut and Daniel~S. Freed.
\newblock The analysis of elliptic families. i. metrics and connections on
  determinant bundles.
\newblock {\em Communications in Mathematical Physics}, 106(1):159--176, 1986.

\bibitem[BGV03]{berline2003heat}
Nicole Berline, Ezra Getzler, and Michele Vergne.
\newblock {\em Heat kernels and Dirac operators}.
\newblock Springer Science \& Business Media, 2003.

\bibitem[Bis86]{Bismut:1986uh}
Jean-Michel Bismut.
\newblock The atiyah-singer index theorem for families of dirac operators: Two
  heat equation proofs.
\newblock {\em Inventiones mathematicae}, 83(1):91--151, 1986.

\bibitem[FH13]{BNablaG}
Daniel~S. Freed and Michael~J. Hopkins.
\newblock Chern-weil forms and abstract homotopy theory, 2013.

\bibitem[Fre16]{EquivDet}
Daniel~S. Freed.
\newblock On equivariant chern-weil forms and determinant lines, 2016.

\bibitem[MQ86]{MATHAI198685}
Varghese Mathai and Daniel Quillen.
\newblock Superconnections, thom classes, and equivariant differential forms.
\newblock {\em Topology}, 25(1):85--110, 1986.

\bibitem[Qui85]{QUILLEN198589}
Daniel Quillen.
\newblock Superconnections and the chern character.
\newblock {\em Topology}, 24(1):89--95, 1985.

\end{thebibliography}

\end{document}